\documentclass[12pt,oneside,reqno]{amsart}
\usepackage{amsmath, amssymb, amsfonts, amsthm, enumerate, enumitem, bm, float}

\usepackage[margin=1in]{geometry}

\newtheorem{prop}{Proposition}[section]
\newtheorem{lemma}[prop]{Lemma}
\newtheorem{thm}[prop]{Theorem}
\newtheorem{cor}[prop]{Corollary}
\newtheorem{ex}[prop]{Example}

\theoremstyle{definition}

\newtheorem{question}[prop]{Question}

\renewcommand{\mod}{\textrm{mod}~}
\DeclareMathOperator{\maxdim}{MaxDim}
\DeclareMathOperator{\mindim}{MinDim}

\begin{document}

\title{On the Intersection Numbers of Finite Groups}

\date{\today}

\author{Kassie Archer}
\author{Humberto Bautista Serrano}
\author{Kayla Cook}
\author{L.-K. Lauderdale}
\author{Yansy Perez}  
\author{Vincent Villalobos}

\begin{abstract}
The \textit{covering number} of a nontrivial finite group $G$, denoted $\sigma(G)$, is the smallest number of proper subgroups of $G$ whose set-theoretic union equals $G$.  In this article, we focus on a dual problem to that of covering numbers of groups, which involves maximal subgroups of finite groups. For a nontrivial finite group $G$, we define the \textit{intersection number} of $G$, denoted $\iota(G)$, to be the minimum number of maximal subgroups whose intersection equals the Frattini subgroup of $G$.  We elucidate some basic properties of this invariant, and give an exact formula for $\iota(G)$ when $G$ is a nontrivial finite nilpotent group.  In addition, we determine the intersection numbers of a few infinite families of non-nilpotent groups.  We conclude by discussing a generalization of the intersection number of a nontrivial finite group and pose some open questions about these invariants.
\end{abstract}

\maketitle
\thispagestyle{empty}


\section{Introduction}

Throughout this article, we consider only finite groups.  In a popular paper of Cohn \cite{COHN1994}, the concept of a covering number of a group was introduced.  The \textbf{covering number} of a nontrivial group $G$, denoted $\sigma(G)$, is the smallest number of proper subgroups of $G$ whose set-theoretic union equals $G$.  For example, the quaternion group of order $8$, denoted $Q_8$, satisfies $\sigma(Q_8)\leq3$ as the union of its three subgroups of order $4$ equals $Q_8$; since each proper subgroup of $Q_8$ has order at most $4$ and contains the identity element, it follows that $\sigma(Q_8)=3$.   Covering numbers are the subject of prior research by numerous authors, and we briefly discuss some of their results below.

A group $G$ has a covering number if and only if $G$ is noncyclic.  Additionally, it is well-known that the value of $\sigma(G)$, when it exists, is at least $3$; those groups $G$ satisfying $\sigma(G)=3$ were considered in \cite{BBM1970,HR1959,SCORZA1926}.  In \cite{COHN1994}, Cohn classified the groups $G$ with $\sigma(G) \in \{4,5,6\}$, and conjectured that there is no group $G$ with $\sigma(G)=7$.  Tomkinson \cite{TOMKINSON1997} proved this conjecture and computed the value of $\sigma(G)$ when $G$ is a noncyclic solvable group. He also inquired about which integers could occur as covering numbers of nontrivial groups, and thought the current evidence suggested that there are no groups $G$ with $\sigma(G) \in \{11,13,15\}$.  However, Abdollahi, Ashraf, and Shaker \cite{AAS2007} proved that the covering number of the symmetric group on 6 symbols is 13; Bryce, Fedri, and Serena \cite{BFS1999} found a linear group with covering number 15.  The remaining case was confirmed by Detomi and Lucchini \cite{DETOMI-LUCCHINI2008}, who proved that there is no group $G$ with $\sigma(G)=11$.  Additionally, the combined results of Garonzi, Kappe, and Swartz \cite{GARONZI2013B,GKS2018} classified every integer less than 130 that is a covering number of some nontrivial group.  More generally, the structure of groups $G$ containing no normal nontrivial subgroup $N$ such that $\sigma(G/N) = \sigma(G)$ was investigated by Detomi and Lucchini \cite{DETOMI-LUCCHINI2008}.  There have also been investigations into the values of $\sigma(G)$ for certain nonsolvable groups $G$ (see \cite{BFS1999, EM2016, GARONZI2013, HOMMES2006, KNPS2016, LUCIDO2003, MAROTI2005, OE2019, SWARTZ2016}), but in general establishing the covering number of a nonsolvable group remains a topic of ongoing research.

The focus of this article is an investigation into a dual problem to that of covering numbers of groups, which involves intersections of maximal subgroups of groups.  Recall that a subgroup $H$ of a group $G$ is \textbf{maximal} if $H$ is a proper subgroup of $G$ and there are no other proper subgroups of $G$ that contain $H$.  The intersection of all maximal subgroups of $G$ is called the \textbf{Frattini subgroup} of $G$ and is denoted by $\Phi(G)$. For a nontrivial group $G$, we define the \textbf{intersection number} of $G$, denoted $\iota(G)$, to be the \emph{minimum} number of maximal subgroups whose intersection is $\Phi(G)$.  As an example, the quaternion group $Q_8$ has three maximal subgroups and the intersection of any two of these is $\Phi(Q_8)$.  Therefore, $\iota(Q_8)=2$ because each maximal subgroup of $Q_8$ properly contains $\Phi(Q_8)$.

For the nontrivial group $G$, there are two group invariants that nicely bound the value of $\iota(G)$.  Let $\mathcal M=\{M_i\}_{i\in I}$ be a family of maximal subgroups of $G$ indexed by a set $I$.  If 
	\[\bigcap_{i \neq j} M_i \subset \bigcap_{i \in I} M_i\]
for all $i \in I$, then $\mathcal M$ is said to be \textbf{irredundant}.  Fernando \cite{FERNANDO2015} defined the \textbf{maximal dimension} of $G$, denoted $\maxdim(G)$, to be the maximal size of an irredundant family of $G$; Garonzi and Lucchini \cite{GL2019} defined the \textbf{minimal dimension} of $G$, denoted $\mindim(G)$, to be the minimal size of an irredundant family of $G$.  If $\iota(G)=k$ with $k\in \mathbb Z^+$ and $M_1,M_2,\ldots,M_k$ are maximal subgroups of $G$ such that
	\[M_1 \cap M_2 \cap \cdots \cap M_k=\Phi(G),\]
then $\{M_i\}_{i=1}^k$ is irredundant.  Therefore,
	\[\mindim(G) \leq \iota(G) \leq \maxdim(G).\]
Burness, Garonzi, and Lucchini \cite{BGL2019} noted that if $G$ is nilpotent, then $\maxdim(G)$ and $\mindim(G)$ coincide; Fernando \cite{FERNANDO2015} prove that $\maxdim(G)$ is equal to the maximal size of a minimal generating set of $G$.  Consequently, $\iota(G)$ is equal to the maximal size of a minimal generating set of $G$, and the results of Proposition~\ref{thm:NIL GROUPS} prove this by first establishing some previously unknown properties of $\iota(G)$.  Moreover, Burness, Garonzi, and Lucchini \cite{BGL2019} independently studied $\iota(G)$ (denoted $\alpha(G)$ in their article) for nonabelian simple groups $G$.  In this article, we consider intersection numbers of other families of groups and conclude by studying a new group invariant.

This article is organized as follows.  In Section~\ref{sec:PRELIMS}, we prove some basic properties about the intersection numbers of nontrivial groups.  The results of Section~\ref{sec:NILPOTENT GROUPS} establish the value of $\iota(G)$, where $G$ is a nontrivial nilpotent group, and in Section~\ref{sec:NON-NILPOTENT}, we consider the intersection numbers of dihedral groups, dicyclic groups, and symmetric groups. In Section~\ref{sec:INN}, we define and then briefly investigate a second group invariant, called the inconjugate intersection number.  Finally, we pose some open questions throughout Section~\ref{sec:OPEN QUESTIONS} about the intersection and inconjugate intersection numbers of nontrivial groups in general.


\section{Preliminary Results}\label{sec:PRELIMS}

The results in this section establish some basic properties about the intersection numbers of nontrivial groups.

\begin{prop}\label{prop:DIRECT PRODUCTS}
If $G_1$ and $G_2$ are nontrivial groups, then $\iota(G_1\times G_2) \leq \iota(G_1)+\iota(G_2)$.
\end{prop}

\begin{proof}
Choose $k,\ell \in \mathbb Z^+$ such that $\iota(G_1)=k$ and $\iota(G_2)=\ell$.  Assume that $M_1,M_2,\ldots,M_k$ and $K_1,K_2,\ldots,K_\ell$ are maximal subgroups of $G_1$ and $G_2$, respectively, such that
		\[M_1\cap M_2\cap \cdots\cap M_k =\Phi(G_1) \quad \text{ and } \quad K_1\cap K_2\cap \cdots\cap K_\ell =\Phi(G_2).\] 
For each $i \in \{1,2, \ldots,k\}$, the subgroup $M_i \times G_2$ is maximal in $G_1\times G_2$.  Additionally, $G_1 \times K_j$ is a maximal subgroup of $G_1 \times G_2$ for each $j \in \{1,2,\ldots,\ell\}$.  It follows that 
		\[\bigcap_{i=1}^k(M_i \times G_2) \cap \bigcap_{j=1}^\ell (G_1 \times K_j)=\Phi(G_1) \times \Phi(G_2)=\Phi(G_1\times G_2),\]
where the last equality holds because $G_1\times G_2$ is a finite group.  Therefore, $\iota(G_1\times G_2) \leq k+\ell$ and the result follows.
\end{proof}

The forthcoming example demonstrates that the inequality given in Proposition~\ref{prop:DIRECT PRODUCTS} can be strict.  

\begin{ex}
If $A_5$ denotes the alternating group on $5$ symbols, then $\iota(A_5\times A_5)< \iota(A_5) + \iota(A_5)$.
\end{ex}

\begin{proof}
The group $A_5$ has three isomorphism types of maximal subgroups, namely the dihedral group of order $10$, the alternating group of order $4$, and the symmetric group of order $3$.  Consider the following two maximal subgroups of $A_5$, each of which is isomorphic to $S_3$:
	\[M_1 = \{1,\ (1,2,3),\ (1,3,2),\ (1,2)(4,5),\ (1,3)(4,5),\ (2,3)(4,5) \}\]
and 
	\[M_2 = \{1,\ (1,2,4),\ (1,4,2),\ (1,2)(3,5),\ (1,4)(3,5),\ (2,4)(3,5) \}.\]
Since 
	\[M_1 \cap M_2 = \{1\} = \Phi(A_5),\]
it follows that $\iota(A_5) \leq 2$.  Moreover,  $\iota(A_5)=2$ because each maximal subgroup of $A_5$ contains $\Phi(A_5)$ as a proper subgroup.  Now consider the group $A_5 \times A_5$ as the permutation group
		\[A_5\times A_5 \cong\langle (1,2,3,4,5)(6,7,8,9,10),(1,3,2)(6,7,8)\rangle.\]
Each of
	\begin{align*}
		K_1& =\langle (1,2,3,4,5)(6,7,8,9,10), (1,3,5)(6,8,10) \rangle,\\
		K_2& =\langle (1,2,3,4,5)(6,8,7,10,9),(1,3,5)(6,7,9) \rangle,
	\end{align*}
and
		\[K_3 =\langle (1,2,3,4,5)(6,9,7,8,10), (1,3,5)(6,7,10) \rangle\]
is a maximal subgroup of this permutation group that is isomorphic to $A_5$.  Since 
	\[K_1 \cap K_2 \cap K_3 = \{1\}=\Phi(A_5 \times A_5),\]
we have that $\iota(A_5 \times A_5) \leq 3$ and the result now follows.
\end{proof}

\begin{prop}\label{lem:QUOTIENT}
Let $N$ be a normal subgroup of the nontrivial group $G$.  If $N \subseteq \Phi(G)$, then $\iota(G/N) = \iota(G)$.
\end{prop}

\begin{proof}
Suppose that $\iota(G)=k$ with $k\in \mathbb Z^+$, and let $M_1,M_2,\ldots,M_k$ be maximal subgroups of $G$ such that
	\[M_1 \cap M_2 \cap \cdots \cap M_k=\Phi(G).\]
By the Fourth Isomorphism Theorem, the maximal subgroups $M/N$ of $G/N$ are in one-to-one correspondence with the maximal subgroups $M$ of $G$.  It follows that 
	\begin{equation}\label{eqn:QUOTIENTS}
		\Phi(G/N)=\Phi(G)/N=(M_1 \cap M_2 \cap \cdots \cap M_k)/N=M_1/N \cap M_2/N \cap \cdots \cap M_k/N,
	\end{equation}
where the last equality also follows holds by the Fourth Isomorphism Theorem. Hence, $\iota(G/N) \leq k$.  On the other hand, if $\iota(G/N) = k$, where $M_1/N, M_2/N, \ldots, M_k/N$ are maximal subgroups of $G/N$, then Equation~\eqref{eqn:QUOTIENTS} implies 
	\[M_1 \cap M_2 \cap \cdots \cap M_k=\Phi(G)\]
Therefore, $\iota(G)=k$ as desired.
\end{proof}

Assume that $G$ is a nontrivial group.  In some sense, Lemma~\ref{lem:QUOTIENT} shows that the invariant $\iota(G)$ respects quotient groups of $G$.  However, $\iota(G)$ does not respect the subgroup structure of $G$.  In particular, for a subgroup $H$ of $G$, it is possible for $\iota(H)$ to exceed $\iota(G)$.

\begin{ex}
Let $S_7$ denote the symmetric group on $7$ symbols.  There exists a proper subgroup $H$ of $S_7$ such that $\iota(S_7)<\iota(H)$.
\end{ex}

\begin{proof}
The order-42 maximal subgroups
	\[M_1=\langle(1,2,3,4,5,6,7),(2,6,5,7,3,4)\rangle\]
and
	\[M_2=\langle (1,2,7,3,5,6,4),(2,6,5,4,7,3)\rangle\]
of $S_7$ satisfy $M_1 \cap M_2=\{1\}$.  Since $\Phi(S_7)=\{1\}$, it follows that $\iota(S_7)=2$.  Now consider the subgroup
	\[H=\langle (1,2),(3,4),(5,6) \rangle\]
of $S_7$, which isomorphic to the elementary abelian $2$-group $\mathbb Z_2 \times \mathbb Z_2 \times \mathbb Z_2$. The intersection of any two maximal subgroups of $H$ has order $2$.  It follows that $\iota(H) =3$ because $\Phi(H)=\{1\}$.  Therefore,
	\[2=\iota(S_7)<\iota(H)=3,\]
as desired.
\end{proof}

It is possible for $\iota(H)$ to exceed $\iota(G)$ even if $H$ is a normal subgroup of $G$.  For example, we assert that 
	\[2=\iota(S_7)<\iota(A_7)=3.\]
With a few basic properties of intersection numbers of groups established, we continue by establishing the exact values of intersection numbers for nontrivial nilpotent groups.


\section{Nilpotent Groups}\label{sec:NILPOTENT GROUPS}

In this section, we establish the value of $\iota(G)$, where $G$ is a nontrivial nilpotent group. To this end, we first consider $p$-groups. Let $p$ be a prime number, and suppose $P$ is a $p$-group. We can consider the elementary abelian $p$-group $P/\Phi(P)$ as a vector space over the field $\mathbb F_p$; the \textbf{rank} of $P$ is the dimension of $P/\Phi(P)$.

\begin{lemma}\label{lem:OMEGA OF EA p-GROUP}
Let $p$ be a prime number. If $P$ is a nontrivial elementary abelian $p$-group of rank $r$, then $\iota(P)=r$.
\end{lemma}

\begin{proof}
If $r=1$, then $P$ is cyclic and has only one maximal subgroup, which is the identity subgroup.  Since $\Phi(P)=\{1\}$, it follows that $\iota(P)=1$.  Assume that $r\geq2$, and for each $i \in \{1,2,\ldots,r\}$, define $H_i$ to be the direct product of $r-1$ factors of $\mathbb Z_p$ and the identity subgroup, where the identity subgroup occurs in the $i$-th factor of $H_i$.  In this case, each $H_i$ is isomorphic to a maximal subgroup of $P$ and $\bigcap_{i=1}^r H_i=\{1\}$.  Because $\Phi(P)=\{1\}$ for every elementary abelian $p$-group, we have that $\iota(P) \leq r$.  

If $r=2$, then $\iota(P)=2$ because each maximal subgroup of $P$ contains $\Phi(P)$ as a proper subgroup.  To prove that $\iota(G)=r$ for all $r \geq 3$, it suffices to show that the intersection of any $r-1$ maximal subgroups of $P$ properly contains $\Phi(P)=\{1\}$.  Towards a contradiction, suppose that $M_1, M_2, \ldots, M_{r-1}$ are maximal subgroups of $P$ such that 
	\[M_1 \cap M_2 \cap \cdots \cap M_{r-1}=\Phi(P).\]
For ease of notation, let $j \in \{1,2, \ldots,r-1\}$ and define $K_j=\bigcap_{i=1}^j M_i$.  Since $|P|=p^r$, we have that
	\[p^r=[P:K_{r-1}]=[P:K_1][K_1:K_2][K_2:K_3] \cdots [K_{r-2}:K_{r-1}]\]
and $p^r$ is the product of the aforementioned $r-1$ indices.  Because $K_1=M_1$ is a maximal subgroup of $P$ and $[P:K_1]=p$, there exists $j \in \{2,3,\ldots,r-2\}$ such that $p^2$ divides $[K_j:K_{j+1}]$.  If $K_j \subseteq M_{j+1}$, then $K_j=K_{j+1}$ and $[K_j:K_{j+1}]=1$, which is impossible because $p^2$ divides $[K_j:K_{j+1}]$.  It follows that $K_j \not\subseteq M_{j+1}$, and thus $P=K_jM_{j+1}$ because $M_{j+1}$ is a maximal subgroup of $P$.  Consequently,
	\[p^r=|P|=\frac{|K_j||M_{j+1}|}{|K_j \cap M_{j+1}|}=\frac{|K_j||M_{j+1}|}{|K_{j+1}|} \geq p^2p^{r-1}=p^{r+1},\]
which is impossible, and hence $\iota(P)=r$.
\end{proof}

\begin{lemma}\label{lem:OMEGA OF p-GROUP}
Let $p$ be a prime number. If $P$ is a nontrivial $p$-group of rank $r$, then $\iota(P)=r$.
\end{lemma}

\begin{proof}
Lemma~\ref{lem:QUOTIENT} and Lemma~\ref{lem:OMEGA OF EA p-GROUP} imply
	\[\iota(P) = \iota(P/\Phi(P))=r,\]
as desired.
\end{proof}

Let $G$ be a nontrivial nilpotent group, and let $\pi(G)$ denote the set of prime numbers that divide the order of $G$.  Since $G$ is the direct product of its Sylow subgroups, we will use Proposition~\ref{lem:OMEGA OF p-GROUP} to establish the exact value of $\iota(G)$.

\begin{thm}\label{thm:NIL GROUPS}
Let $G$ be a nontrivial nilpotent group.  If $\pi(G)=\{p_1,p_2,\ldots,p_k\}$, then
	\[\iota(G)=\sum_{j=1}^k \iota(P_j)=\sum_{j=1}^k r_j,\]
where $P_j$ is a Sylow $p_j$-subgroup of $G$ with rank $r_j$ for each $j \in \{1,2, \ldots, k\}$.
\end{thm}

\begin{proof}
Since $G\cong P_1 \times P_2 \times \cdots \times P_k$ is nilpotent, all maximal subgroups of $G$ have prime index in $G$.  For each $j \in \{1,2, \ldots,k\}$, the maximal subgroups of $G$ with index $p_j$ are in one-to-one correspondence with the maximal subgroups of $P_j$. The result now follows from Proposition~\ref{lem:OMEGA OF p-GROUP} and induction on $k$.
\end{proof}

The forthcoming corollary of Theorem~\ref{thm:NIL GROUPS} highlights a major difference between the invariants $\iota(G)$ and $\sigma(G)$.  In particular, Tomkinson \cite{TOMKINSON1997} proved that there is no group with covering number $7$, and the subsequent theorem identifies all covering numbers less than 130.

\begin{thm}\emph{(Garonzi \cite{GARONZI2013B}; Garonzi, Kappe and Swartz \cite{GKS2018})}
The following list identifies exactly which integers less than 130 that are not covering numbers: $1$, $2$, $7$, $11$, $19$, $21$, $22$, $25$, $27$, $34$, $35$, $37$, $39$, $41$, $43$, $45$, $47$, $49$, $51$, $52$, $53$, $55$, $56$, $58$, $59$, $61$, $66$, $69$, $70$, $75$, $76$, $77$, $78$, $79$, $81$, $83$, $87$, $88$, $89$, $91$, $93$, $94$, $95$, $96$, $97$, $99$, $100$, $101$, $103$, $105$, $106$, $107$, $109$, $111$, $112$, $113$, $115$, $116$, $117$, $118$, $119$, $120$, $123$, $124$, $125$.
\end{thm}

The corollary below proves that for each $k \in \mathbb Z^+$, there exists a group $G$ such that $\iota(G)=k$.

\begin{cor}
If $G$ is a nontrivial cyclic group, then $\iota(G)=|\pi(G)|$.
\end{cor}

\begin{proof}
Assume that $|G|=p_1^{e_1}p_2^{e_2}\cdots p_k^{e_k}$ is the prime factorization of $|G|$, where $e_j \in \mathbb Z^+$ for all $j\in\{1,2,\ldots,k\}$. By the Fundamental Theorem of Finite Abelian Groups,
	\[G \cong \mathbb Z_{p_1^{e_1}} \times \mathbb Z_{p_2^{e_2}} \times \cdots \times \mathbb Z_{p_k^{e_k}}\]
and Lemma~\ref{lem:OMEGA OF EA p-GROUP} implies that $\iota(\mathbb Z_{p_j^{e_j}})=1$ for each  $j\in\{1,2,\ldots,k\}$. The result now follows from Theorem~\ref{thm:NIL GROUPS}.
\end{proof}

With the values of $\iota(G)$ established for all nontrivial nilpotent groups $G$, it is natural to investigate this invariant for non-nilpotent groups. In the next section, we consider the intersection numbers of dihedral groups, dicyclic groups, and symmetric groups.


\section{Non-Nilpotent Groups}\label{sec:NON-NILPOTENT}

In this section, we investigate the intersection numbers of three infinite families of non-nilpotent groups.  Let $D_{2n}$ denote the dihedral group of order $2n$, where $n \geq 3$. It is well-known that $D_{2n}$ is nilpotent exactly when $n$ is a power of $2$. If $n$ is a power of $2$, then $D_{2n}$ is of rank $2$, and Proposition~\ref{lem:OMEGA OF p-GROUP} implies that $\iota(D_{2n})=2$.  To establish the values of $\iota(D_{2n})$ when $n$ is not a power of $2$, we utilize the following presentation of the dihedral group:
	\[D_{2n} = \left\langle r,s: r^n=1=s^2,\ r^{-1}=srs \right\rangle.\]
	
\begin{prop}\label{prop:DIHEDRAL}
If $n=p_1^{e_1}p_2^{e_2}\cdots p_k^{e_k}$ is the prime factorization of $n$, where $e_1,e_2,\ldots, e_k \in \mathbb Z^+$, then $\iota(D_{2n})=k+1$.
\end{prop}

\begin{proof}
If $n$ is a power of $2$, then $\iota(D_{2n})=2$ by Proposition~\ref{lem:OMEGA OF p-GROUP}.  Thus, we assume that $n$ is not a power of $2$ for the remainder of the proof.  Every maximal subgroup of $D_{2n}$ is either:
	\begin{enumerate}[label=$($\alph*$)$]\setlength\itemsep{2pt}
		\item the cyclic subgroup $\langle r \rangle$ of index $2$; or
		\item the dihedral subgroup $\langle r^{p_i}, r^{j_i}s\rangle$ of index $p_i$ for some $i \in \{1,2, \ldots,k\}$ and $j_i \in \{0,1,\ldots,p_i-1\}$.
	\end{enumerate}
The intersection of these $1+p_1+p_2+\cdots+p_k$ maximal subgroups of $D_{2n}$ is $\Phi(D_{2n})=\big\langle r^{p_1p_2\cdots p_k}\big\rangle$. Notice that the intersection of $\langle r \rangle$ with $\bigcap_{i=1}^k \langle r^{p_i},s \rangle$ is equal to $\Phi(D_{2n})$, and as a result $\iota(D_{2n}) \leq k+1$. 

Towards a contradiction, suppose $M_1,M_2,\ldots,M_k$ are maximal subgroups of $D_{2n}$ that satisfy
	\[M_1 \cap M_2 \cap \cdots \cap M_k = \Phi(D_{2n}).\]
Further, assume that there exists $i \in \{1,2,\ldots,k\}$ such that $\langle r^{p_i},r^{j_i}s \rangle \not \in \{M_1,M_2,\ldots, M_k\}$ for all $j_i \in \{0,1,\ldots,p_i-1\}$.  In this case, 
	\[ \big\langle r^{\frac{p_1p_2\cdots p_k}{p_i}} \big\rangle <M_1 \cap M_2 \cap \cdots \cap M_k = \Phi(D_{2n}) =\big\langle r^{p_1p_2\cdots p_k}\big\rangle,\]
which is impossible.  Therefore, after a possible relabeling, assume that $M_i=\langle r^{p_i}, r^{j_i}s \rangle$.  The Chinese Remainder Theorem guarantees that the system of linear congruences given by
	\begin{align*}
		\ell & \equiv  j_1  \mod {p_1^{e_1}}\\
		\ell &  \equiv  j_2  \mod {p_2^{e_2}}\\
		& \, \ \vdots \\
		\ell & \equiv  j_k  \mod {p_k^{e_k}}
	\end{align*}
has a unique solution modulo $n=p_1^{e_1}p_2^{e_2}\cdots p_k^{e_k}$.  Hence, 
	\[r^\ell s \in M_1 \cap M_2 \cap \cdots \cap M_k=\Phi(D_{2n})=\big\langle r^{p_1p_2\cdots p_k}\big\rangle,\]
a final contradiction.  The result now follows.
\end{proof}

Next, we turn our attention to the dicyclic groups.  The dicyclic group $Q_{4n}$ of order $4n$, where $n \geq 2$, has presentation 
	\[Q_{4n} = \left\langle x,y: x^{2n}=1, x^n=y^2, x^{-1}=y^{-1}xy \right\rangle.\]
If $n$ is a power of $2$, then $Q_{4n}$ is also known as the generalized quaternion group.  In this case, $Q_{4n}$ is a nilpotent group of rank $2$, and Proposition~\ref{lem:OMEGA OF p-GROUP} implies that $\iota(Q_{4n})=2$.  The following proposition also considers the non-nilpotent group $Q_{4n}$ (i.e., considers the case when $n$ is not a power of $2$).  Its proof is omitted do to the extreme similarities to the proof of Proposition~\ref{prop:DIHEDRAL}.

\begin{prop}
If $n=p_1^{e_1}p_2^{e_2}\cdots p_k^{e_k}$ is the prime factorization of $n$, where $e_1,e_2,\ldots, e_k \in \mathbb Z^+$, then $\iota(Q_{4n})=k+1$.
\end{prop}

To conclude this section, we give an upper bound on the value of $\iota(S_n)$, where $S_n$ denotes the symmetric group on $n$ symbols. 

\begin{prop}\label{prop:SYMMETRIC BOUND}
If $n \geq 4$ is an integer, then $\iota(S_n) \leq \left\lfloor\frac{n+8}{4}\right\rfloor$.
\end{prop}

\begin{proof}
The O'Nan-Scott Theorem (see \cite{LIEBECK-PRAEGER-SAXL1988} and \cite{SCOTT1980}) classifies the maximal subgroups of $S_n$; it states that $S_k \times S_{n-k}$ is isomorphic to a maximal subgroup of $S_n$ provided $1 \leq k < \frac n2$.  If $a=\left\lceil \frac{n+1}{2}\right \rceil$ and $b=\left\lfloor \frac {n-1}{2}\right\rfloor$, then $1 \leq b < \frac n2$ and $a+b=n$.  Therefore, $S_a \times S_b$ is isomorphic to a maximal subgroup of $S_n$.  We claim that there exist $\left\lfloor\frac{n+8}{4}\right\rfloor$ subgroups of $S_n$ that are isomorphic to $S_a\times S_b$ whose intersection is $\Phi(S_n)=\{1\}$.  We will use the permutation $\sigma=(1,2,\ldots,n)$ to build these maximal subgroups of $S_n$, and the following construction depends on the parity of $n$.

First assume that $n$ is an odd integer.  Let $\ell_m$ denote the integer
	\[\ell_m=\begin{cases}
		0 & \text{if } \gcd(n,j)=1\\
		\left\lfloor \frac{m\cdot \gcd(n,j)}{n} \right\rfloor & \text{if } \gcd(n,j) \geq 2,
	\end{cases}\]
where $m \in \{1,2,\ldots,n-1\}$. For each $j\in \{1, 2, \ldots, \left\lfloor\frac{n+8}{4}\right\rfloor-1\}$, define the permutations $\alpha_j$ and $\beta_j$ on $\{1,2,\ldots,n\}$ by 
	\[\alpha_j=\big(1,\sigma^{{2j}+\ell_1}(1),\sigma^{{4j}+\ell_2}(1),\ldots, \sigma^{2(a-1)j+\ell_{a-1}}(1)\big)\big(\sigma^{2aj+\ell_a}(1),\ldots, \sigma^{2(n-1)j+\ell_{n-1}}(1)\big)\]
and 
	\[\beta_j=\big(1,\sigma^{2j+\ell_1}(1)\big)\big(\sigma^{2aj+\ell_a}(1),\sigma^{2(a+1)j+\ell_{a+1}}(1)\big).\]
Observe that $\alpha_j$ is a product of a cycle of length $a$ and a cycle of length $b$, and $\beta_j$ is a product of $2$ transpositions. Moreover, the supports of the cycles that appear in the cycle decomposition of  $\alpha_j$ or $\beta_j$ are disjoint.  Since disjoint cycles commute and $S_k=\langle(1,2 \ldots,k),(1,2)\rangle$ for all $k \in \mathbb Z^+$, we have $\langle \alpha_j,\beta_j\rangle \cong S_a \times S_b$ and $\langle \alpha_j,\beta_j\rangle$ is a maximal subgroup of $S_n$ for each $j\in\{1, 2, \ldots, \left\lfloor\frac{n+8}{4}\right\rfloor-1\}$.  Additionally, define the permutations $\alpha_0$ and $\beta_0$ on $\{1,2,\ldots,n\}$ by
	\[\alpha_0=\big(1,\sigma(1),\sigma^2(1),\ldots, \sigma^{a-1}(1)\big)\big(\sigma^a(1),\sigma^{a+1}(1),\ldots, \sigma^{n-1}(1)\big),\]
and
	\[\beta_0=\big(1,\sigma(1)\big)\big(\sigma^a(1),\sigma^{a+1}(1)\big).\]
Notice that $\langle \alpha_0, \beta_0\rangle$ is also a maximal subgroup of $S_n$ because $\langle \alpha_0, \beta_0\rangle \cong S_a\times S_b$.

For a contradiction, suppose
	\[\bigcap_{j=0}^{\left\lfloor\frac{n+8}{4}\right\rfloor-1} \langle \alpha_j,\beta_j\rangle \neq \{1\},\]
and choose
	\[\rho \in \bigcap_{j=0}^{\left\lfloor\frac{n+8}{4}\right\rfloor-1} \langle \alpha_j,\beta_j\rangle\]
such that $\rho(k)=\ell$ for some $k, \ell \in \{1,2,\ldots,n\}$ and $k \neq \ell$.  In this case, the integers $k$ and $\ell$ must lie in the support of the same cycle appearing the cycle decomposition of $\alpha_j$ for all $j \in \{0,1,\ldots,\lfloor\frac{n+8}{4}\rfloor-1\}$.  However, by construction there exists $j \in \{0,1,\ldots,\left\lfloor\frac{n+8}{4}\right\rfloor-1\}$ such that $k$ and $\ell$ lie in different cycles in the cycle decomposition of $\alpha_j$, a contradiction.  Therefore,
		\[\bigcap_{j=0}^{\left\lfloor\frac{n+8}{4}\right\rfloor-1} \langle \alpha_j,\beta_j\rangle = \{1\} = \Phi(S_n)\]
and $\iota(S_n) \leq \left\lfloor\frac{n+8}{4}\right\rfloor$ when $n$ is odd.
	
Now assume that $n$ is even.  If $n\in\{4,6\}$, then it is easy to verify that $\iota(S_n) \leq \left\lfloor\frac{n+8}{4}\right\rfloor =3$. Thus, we assume that $n \geq 8$ for the remainder of the proof.  For each $j \in \{1,n-1\}$, define the permutations $\alpha_j$, $\beta_j$ and $\gamma_j$ on $\{1,2,\ldots,n\}$ by
	\[\alpha_j=\big(1,\sigma^{j}(1),\sigma^{2j}(1),\ldots, \sigma^{(b-1)j}(1)\big)\big(\sigma^{bj}(1),\sigma^{(b+1)j}(1),\ldots, \sigma^{(n-1)j}(1)\big),\]
	\[\beta_j=\big(1,\sigma^{j}(1)\big) \quad \text{ and } \quad \gamma_j=\big(\sigma^{bj}(1),\sigma^{(b+1)j}(1)\big).\]
Additionally, let $k_n$ be the largest integer less than $\frac12(\frac n2+3)$ that is relatively prime to $\frac n2+3$, and let 
	\[\ell_m=\begin{cases}
		0 & \text{if } \gcd(n,j)=1\\
		\left\lfloor \frac{m\cdot \gcd(n,j)}{n} \right\rfloor & \text{if } \gcd(n,j) \geq 2,
	\end{cases}\]
where $m \in \{1,2,\ldots,n-1\}$.  For each
	\[j \in\left\{k_n, k_n+1,k_n+2, \ldots, k_n+\left\lfloor\frac{n-4}{4}\right\rfloor\right\},\]
define the permutations $\alpha_j$, $\beta_j$ and $\gamma_j$ on $\{1,2,\ldots,n\}$ by
	\[\alpha_j=\big(1,\sigma^{j+\ell_1}(1),\sigma^{2j+\ell_2}(1),\ldots, \sigma^{(a-1)j+\ell_{a-1}}(1)\big)\big(\sigma^{aj+\ell_a}(1),\ldots, \sigma^{(n-1)j+\ell_{n-1}}(1)\big),\]
	\[\beta_j=\big(1,\sigma^{j+\ell_1}(1)\big)\quad \text{ and } \quad \gamma_j=\big(\sigma^{aj+\ell_a}(1),\sigma^{(a+1)j+\ell_{a+1}}(1)\big).\]
A similar argument to that above proves that $\langle \alpha_j,\beta_j,\gamma_j\rangle \cong S_a \times S_b$ and
	\[\bigcap_j \langle \alpha_j,\beta_j,\gamma_j\rangle = \{1\} = \Phi(S_n).\]
Since
	\[\left\lfloor\frac{n-4}{4}\right\rfloor=\left\lfloor\frac{n+8}{4}\right\rfloor-3,\]
the cardinality of 
	\[\{1,n-1\}\cup\left\{k_n, k_n+1,k_n+2, \ldots, k_n+\left\lfloor\frac{n-4}{4}\right\rfloor\right\}\]
equals $\left\lfloor\frac{n+8}{4}\right\rfloor$.  It follows that $\iota(S_n) \leq \left\lfloor\frac{n+8}{4}\right\rfloor$ when $n$ is even, as desired.
\end{proof}

As an example of the proof of Proposition~\ref{prop:SYMMETRIC BOUND}, we assume that $n=11$ and will construct $4=\left\lfloor\frac{11+8}{4}\right\rfloor$ maximal subgroups of $S_{11}$ whose intersection is the identity subgroup to prove that $\iota(S_{11}) \leq 4$.  If $n=11$, then $a=6$, $b=5$, and $\sigma=(1,2,\ldots,11)$; notice that $\ell_{2m}=0$ for all $m\in \{1,2,\ldots,8\}$.  For each $j \in \{1,2,3\}$, the permutations $\alpha_j$ and $\beta_j$ on $\{1,2,\ldots, 11\}$ are defined by
	\[\alpha_j=\big(1,\sigma^{2j}(1),\sigma^{4j}(1),\sigma^{6j}(1),\sigma^{8j}(1), \sigma^{10j}(1)\big)\big(\sigma^{12j}(1),\sigma^{14j}(1),\sigma^{16j}(1),\sigma^{18j}(1),\sigma^{20j}(1)\big)\]
and 
	\[\beta_j=\big(1,\sigma^{2j}(1)\big)\big(\sigma^{12j}(1),\sigma^{14j}(1)\big);\]
that is,
	\begin{eqnarray*}
		\alpha_1 & = & \big(1,\sigma^2(1),\sigma^4(1),\sigma^6(1),\sigma^8(1), \sigma^{10}(1)\big)\big(\sigma^{12}(1),\sigma^{14}(1),\sigma^{16}(1),\sigma^{18}(1),\sigma^{20}(1)\big)\\
		& = & (1,3,5,7,9,11)(2,4,6,8,10), \\ [6pt]
		\beta_1 & = & (1,3)(2,4),\\[6pt]
		\alpha_2 & = & \big(1,\sigma^4(1),\sigma^8(1),\sigma^{12}(1),\sigma^{16}(1), \sigma^{20}(1)\big)\big(\sigma^{24}(1),\sigma^{28}(1),\sigma^{32}(1),\sigma^{36}(1),\sigma^{40}(1)\big)\\
		& = &(1,5,9,2,6,10)(3,7,11,4,8),\\ [6pt]
		\beta_2 & = & (1,5)(3,7),\\ [6pt]
		\alpha_3 & = & \big(1,\sigma^6(1),\sigma^{12}(1),\sigma^{18}(1),\sigma^{24}(1), \sigma^{30}(1)\big)\big(\sigma^{36}(1),\sigma^{42}(1),\sigma^{48}(1),\sigma^{54}(1),\sigma^{60}(1)\big)\\
		& = & (1,7,2,8,3,9)(4,10,5,11,6), \\[6pt]
		\beta_3 & = & (1,7)(4,10).
	\end{eqnarray*}
Also, the permutations $\alpha_0$ and $\beta_0$ on $\{1,2,\ldots,11\}$ are defined by
	\begin{eqnarray*}
		\alpha_0 & = & \big(1,\sigma(1),\sigma^2(1),\sigma^3(1),\sigma^4(1),\sigma^5(1)\big)\big(\sigma^6(1),\sigma^7(1),\sigma^8(1),\sigma^9(1),\sigma^{10}(1)\big) \\ 
		& = & (1,2,3,4,5,6)(7,8,9,10,11)
	\end{eqnarray*}
and
	\[\beta_0 \ = \ \big(1,\sigma(1)\big)\big(\sigma^6(1),\sigma^7(1)\big)\  = \ (1,2)(7,8).\]
The subgroups $\langle\alpha_1,\beta_1\rangle$, $\langle\alpha_2,\beta_2\rangle$, $\langle\alpha_3,\beta_3\rangle$ and $\langle\alpha_0,\beta_0\rangle$ of $S_{11}$ are each isomorphic to $S_6\times S_5$, and thus are maximal subgroups of $S_{11}$.  Since 
	\[ \langle\alpha_0,\beta_0\rangle \cap \langle\alpha_1,\beta_1\rangle \cap \langle\alpha_2,\beta_2\rangle \cap \langle\alpha_3,\beta_3\rangle =\{1\}\]
and $\Phi(S_{11})=\{1\}$, we see that $\iota(S_{11}) \leq 4$, as desired.  We remark that this bound is not best possible; for example, we assert that $\iota(S_{11})=3$, as seen in Table~\ref{table:INTERSECTION NUMBERS OF Sn}. The exact values of $\iota(S_n)$ can also be seen in Table~\ref{table:INTERSECTION NUMBERS OF Sn} for $n \in\{2,3,\ldots,10\}$.

\begin{table}[H]
{\renewcommand{\arraystretch}{1.4}
	\begin{tabular}{|c|c|c|c|c|c|c|c|c|c|c|} \hline
		$n$ & 2 & 3 & 4 & 5 & 6 & 7 & 8 & 9 & 10 & 11\\ \hline
		$\big\lfloor\frac{n+8}{4}\big\rfloor$ & 2 & 2 & 3 & 3 & 3 & 3 & 4 & 4 & 4 & 4\\ \hline
		$\iota(S_n)$ & $1$ & $2$ & $3$ & $3$ & $3$ & $2$ & $3$ & $3$ & $3$ & $3$\\\hline
	\end{tabular}}
\caption{Intersection numbers of $S_n$ for $n \in \{2,3,\ldots,11\}$}
\label{table:INTERSECTION NUMBERS OF Sn}
\end{table}


\section{Inconjugate Intersection Number}\label{sec:INN}

Our work above that established the intersection numbers of certain nontrivial groups has compelled us to define a second group invariant, which we call the \textit{inconjugate intersection number}.  Recall that two subgroups $H$ and $K$ of a group $G$ are \textbf{conjugate} subgroups if there exists $g \in G$ such that $gHg^{-1}=K$.  If two subgroups of $G$ are not conjugate, then we say they are \textbf{inconjugate} subgroups of $G$.  For a nontrivial group $G$, we define the \textbf{inconjuage intersection number} of $G$, denoted $\hat\iota(G)$, to be the \emph{minimum} number of inconjugate maximal subgroups whose intersection equals $\Phi(G)$; set $\hat\iota(G) = \infty$ if $G$ has no inconjugate maximal subgroups whose intersection equals $\Phi(G)$.  

As an example, consider the Frobenius group of order 42, denoted by $\mathbb Z_7 \rtimes \mathbb Z_6$.  The intersection of any two inconjugate maximal subgroups is a nontrivial $p$-group, where $p \in \{2,3,7\}$. Since $\Phi(\mathbb Z_7 \rtimes \mathbb Z_6)=\{1\}$ and the intersection of any three inconjugate maximal subgroups is the identity subgroup, we have that $\hat\iota(\mathbb Z_7 \rtimes \mathbb Z_6)=3$. Additionally, for a group $G$, $\hat\iota(G)$ is not always finite.  For instance, consider the Frobenius group of order $20$, denoted by $\mathbb Z_5 \rtimes \mathbb Z_4$.  The intersection of any two inconjugate maximal subgroups contains a subgroup of order $2$.  However, $\Phi(\mathbb Z_5 \rtimes \mathbb Z_4)=\{1\}$, and as a result $\hat\iota(\mathbb Z_5 \rtimes \mathbb Z_4)=\infty$.  

\begin{prop}\label{INN NIL}
If $G$ is a nontrivial nilpotent group, then $\iota(G)=\hat\iota(G)$. 
\end{prop}

\begin{proof}
Every maximal subgroup $M$ in a nilpotent group $G$ is normal; that is, $gMg^{-1}=M$ for all $g \in G$. The result now follows.
\end{proof}

Proposition~\ref{prop:DIHEDRAL} (i.e., the dihedral group $D_{2n}$) demonstrates that the converse to Proposition~\ref{INN NIL} is false because the maximal subgroups of $D_{2n}$ used in the proof of Proposition~\ref{prop:DIHEDRAL} to establish that $\iota(D_{2n}) = k+1$ are inconjugate.

\begin{prop}\label{INN OF D2n}
If $n=p_1^{e_1}p_2^{e_2}\cdots p_k^{e_k}$ is the prime factorization of $n$, where $e_1,e_2,\ldots, e_k \in \mathbb Z^+$, then $\hat\iota(D_{2n})=k+1$.
\end{prop}

Assume that $p \geq 5$ is a prime number, and let $r \in \mathbb Z_p^\times$ have order $p-1$.  The group defined by the presentation
	\[F_p=\langle x,y:x^p=1=y^{p-1},y^{-1}xy=x^r\rangle\]
is called a \textbf{Frobenius group}.  In this case, $F_p\cong \mathbb Z_p \rtimes \mathbb Z_{p-1}$ and $F_p$ has order $p(p-1)$.

\begin{prop}\label{INN OF Fp}
Let $p\geq 5$ be a prime number. If $F_p\cong \mathbb Z_p \rtimes \mathbb Z_{p-1}$ denotes the Frobenius group of order $p(p-1)$, then $\iota(F_p)=2$ and 
	\[\hat\iota(F_p)=\begin{cases} k+1 & \text{if } p-1 \text{ is a square-free integer}\\  \infty & \text{otherwise},\end{cases}\]
where $k$ is the number of distinct prime divisors of $p-1$.
\end{prop}

\begin{proof}
Write $q=p-1$ and suppose $q=q_1^{e_1}q_2^{e_2}\cdots q_k^{e_k}$ is the prime factorization of $q$, where $e_1,e_2,\ldots, e_k \in \mathbb Z^+$. Every maximal subgroup of $F_p$ is isomorphic to:
	\begin{enumerate}[label=$($\alph*$)$]\setlength\itemsep{2pt}
		\item a cyclic subgroup $\mathbb Z_q$ of index $p$; or
		\item a semi-direct product $\mathbb Z_p \rtimes \mathbb Z_{q/q_i}$ of index $q_i$ for some $i \in \{1,2, \ldots,k\}$.
	\end{enumerate}
The intersection of these $p+k$ maximal subgroups of $F_p$ is $\Phi(F_p)=\{1\}$. Moreover, the intersection of any two distinct conjugate copies of $\mathbb Z_q$ equals $\{1\}$.  It follows that $\iota(F_p)=2$ because each maximal subgroup of $F_p$ contains $\Phi(F_p)$ as a proper subgroup.

Notice that $F_p$ has $1+k$ conjugacy classes of maximal subgroups: $1$ conjugacy class of size $p$ (containing the $p$ conjugates of $\mathbb Z_q$) and $k$ conjugacy classes of size $1$ (each containing the normal subgroup $\mathbb Z_p \rtimes \mathbb Z_{q/q_i}$ of $F_p$ for some $i \in \{1,2,\ldots,k\}$).  Let $\mathcal M$ denote this conjugacy class of size $p$. First, suppose $q$ is a square-free integer so that $e_i=1$ for all $i \in \{1,2,\ldots,k\}$.  The intersection of
	\begin{equation}\label{eqn:Fp}
		\bigcap_{i=1}^k (\mathbb Z_p \rtimes \mathbb Z_{q/q_i})= \mathbb Z_p,
	\end{equation}
and any maximal subgroup in $\mathcal M$ is the identity subgroup.  Thus, $\hat\iota(F_p) \leq k+1$.    Equation~\ref{eqn:Fp} implies that the intersection of any inconjugate maximal subgroups of $F_p$ whose intersection is $\Phi(F_p)=\{1\}$ must contain a maximal subgroup in $\mathcal M$.  If $Q=\{q_1,q_2,\ldots,q_k\}$, then for each $M \in \mathcal M$ and $j\in \{1,2,\ldots,k\}$,
	\[M \cap \bigcap_{Q\backslash\{q_j\}} (\mathbb Z_p \rtimes \mathbb Z_{q/q_i})\]
contains a subgroup isomorphic to $\mathbb Z_{q_j}$.  It follows that $\hat\iota(F_p)>k$, and thus $\hat\iota(F_p)= k+1$ provided $q$ is square-free.  Now, suppose $q$ is not square-free; assume that $e_i\geq2$ for some $i\in \{1,2,\ldots,k\}$.  The intersection of any inconjuage maximal subgroups will contain a subgroup isomorphic to $\mathbb Z_{q_i}$.  Since $\Phi(F_p)=\{1\}$, it follows that $\hat\iota(F_p)=\infty$.
\end{proof}


\section{Discussion and Open Questions}\label{sec:OPEN QUESTIONS}

In this section, we pose four open questions that involve the intersection number and the inconjugate intersection number of the group $G$. Recall that if $G$ is a nontrivial nilpotent group, then $\iota(G)=\hat\iota(G)$ by Proposition~\ref{INN NIL}.  However, Propositions~\ref{INN OF D2n} and~\ref{INN OF Fp} established the existence of infinity families of non-nilpotent, solvable groups $G$ such that $\iota(G)=\hat\iota(G)$ and $\iota(G)<\hat\iota(G)<\infty$, repsectively. Naturally, we ask the following question.

\begin{question}
What structural characteristics of $G$ imply that $\iota(G)=\hat\iota(G)$?
\end{question}

Of course, $\iota(G) \leq \hat\iota(G)$ for all nontrivial groups $G$, and Proposition~\ref{INN OF Fp} established an infinite family of groups $G$ that demonstrates $\hat\iota(G)$ is not always finite. 

\begin{question}
What structural characteristics of $G$ imply that $\hat\iota(G)=\infty$?
\end{question}

Recall that $\pi(G)$ denotes the set of prime numbers that divide the order of $G$, and let $m(G)$ denote the set of maximal subgroups of $G$.  In \cite{LAUDERDALE2013}, the fourth author of this article proved the following result about the relationship between the sizes of $\pi(G)$ and $m(G)$.  

\begin{thm}\emph{\cite{LAUDERDALE2013}}\label{LOWERBOUND}
If $G$ is a finite group, then $|m(G)| \geq |\pi(G)|$.  Furthermore, equality holds if and only if $G$ is cyclic.
\end{thm}

We believe that there is an analogous result involving $|m(G)|$ and $\iota(G)$, where $G$ is a nontrivial group. It is clear that $\iota(G) \leq |m(G)|$, and the results of this article seem to indicate that the following question would have an affirmative answer. 

\begin{question}
Does the inequality $\iota(G)<|m(G)|$ hold for all noncyclic groups $G$?
\end{question}

We are also interested in the relationship between $|\pi(G)|$ and $\hat\iota(G)$ for nontrivial solvable groups $G$.  For instance, assume that $G$ is a solvable group satisfying $\hat\iota(G)=2$; let $M_1$ and $M_2$ are inconjugate maximal subgroups of $G$ such that $M_1 \cap M_2 = \Phi(G)$. Under these assumptions, $G=M_1M_2$ and thus
	\begin{equation}\label{eqn:OMEGA_I IS 2}
		\frac{|G|}{|M_1|} \cdot \frac{|G|}{|M_2|} = \frac{|G|}{|\Phi(G)|}.
	\end{equation}
Since maximal subgroups of nontrivial solvable groups have prime power index and $\pi(G)=\pi(G/\Phi(G))$, Equation~\eqref{eqn:OMEGA_I IS 2} implies that $|\pi(G)|\leq 2$.  Therefore, $|\pi(G)| \leq \hat\iota(G)$ provided $G$ is a solvable group satisfying $\hat\iota(G)=2$.  Consequently, we ask the following question.

\begin{question}
Does the inequality $|\pi(G)| \leq \hat\iota(G)$ hold for all solvable groups $G$?
\end{question}

We claim that it is possible for $|\pi(G)|$ to exceed the intersection number $\iota(G)$.  For example, reconsider the Frobenius group of order 42, denoted by $\mathbb Z_7 \rtimes \mathbb Z_6$.  This group is solvable, and we identify it with the following permutation group:
	\[\mathbb Z_7 \rtimes \mathbb Z_6\cong \langle (1,2,3)(4,5,6), (2,4)(3,7)(5,6) \rangle.\]
The conjugate order-6 maximal subgroups
	\[M_1=\langle(1,2,5,4,3,7)\rangle \quad \text{ and } \quad M_2=\langle (1,3,4,7,5,6)\rangle\]
of $\mathbb Z_7 \rtimes \mathbb Z_6$ satisfy $M_1 \cap M_2=\{1\}$.  Since $\Phi(\mathbb Z_7 \rtimes \mathbb Z_6)=\{1\}$, it follows that $\iota(\mathbb Z_7 \rtimes \mathbb Z_6)=2$ and
	\[2=\iota(\mathbb Z_7 \rtimes \mathbb Z_6)<|\pi(\mathbb Z_7 \rtimes \mathbb Z_6)|=3,\]
as desired.


\section*{Acknowledgements} The authors would like to thank the University of Texas at Tyler's Office of Sponsored Research and Center for Excellence in Teaching and Learning for their support in conducting this research. The awards from these offices supported the research conducted for this article by two faculty members, Kassie Archer and L.-K. Lauderdale, together with undergraduate students Yansy Perez and Vincent Villalobos, and graduate students Humberto Bautista Serrano and Kayla Cook in the fall of 2017.  The fourth author of this article would also like to thank her Ph.D. advisor Alexandre Turull; the initial conversations about what we now call the intersection number began in his office.


\bibliographystyle{plain}
\bibliography{References}

\begin{thebibliography}{10}

\bibitem{AAS2007}
A.~Abdollahi, F.~Ashraf, and S.~M. Shaker.
\newblock The symmetric group of degree six can be covered by 13 and no fewer
  proper subgroups.
\newblock {\em Bull. Malays. Math. Sci. Soc. (2)}, 30(1):57--58, 2007.

\bibitem{BBM1970}
M.~Bruckheimer, A.~C. Bryan, and A.~Muir.
\newblock Groups as the union of three subgroups.
\newblock {\em Am Math Mon.}, 77(1):52--58, 1970.

\bibitem{BFS1999}
R.~A. Bryce, V.~Fedri, and L.~Serena.
\newblock Subgroup coverings of some linear groups.
\newblock {\em Bull. Austral. Math. Soc.}, 60(2):227--238, 1999.

\bibitem{BGL2019}
Timothy~C. Burness, Martino Garonzi, and Andrea Lucchini.
\newblock On the minimal dimension of a finite simple group.
\newblock ar{X}iv:1903.09607v1, March 2019.

\bibitem{COHN1994}
J.~H.~E. Cohn.
\newblock On $n$-sum groups.
\newblock {\em Math. Scand.}, 75:44--58, 1994.

\bibitem{DETOMI-LUCCHINI2008}
Eloisa Detomi and Andrea Lucchini.
\newblock On the structure of primitive $n$-sum groups.
\newblock {\em Cubo}, 10(3):195--210, 2008.

\bibitem{EM2016}
Michael Epstein and Spyros~S. Magliveras.
\newblock The covering number of ${M}_{24}$.
\newblock {\em J. Algebra Comb. Discrete Appl.}, 3(3):155--158, 2016.

\bibitem{FERNANDO2015}
Ravi Fernando.
\newblock On an inequality of dimension-like invariants for finite groups.
\newblock ar{X}iv:1502.00360v1, Feb 2015.

\bibitem{GARONZI2013}
Martino Garonzi.
\newblock Covering monolithic groups with proper subgroup.
\newblock {\em Int. J. Group Theory}, 2(1):131--144, 2013.

\bibitem{GARONZI2013B}
Martino Garonzi.
\newblock Finite groups that are the union of at most 25 proper subgroups.
\newblock {\em J. Algebra Appl.}, 12(4):1350002, 2013.

\bibitem{GKS2018}
Martino Garonzi, Luise-Charlotte Kappe, and Eric Swartz.
\newblock On integers that are covering numbers of groups.
\newblock https://arxiv.org/abs/1805.09047, November 2018.

\bibitem{GL2019}
Martino Garonzi and Andrea Lucchini.
\newblock Maximal irredundant families of minimal size in the alternating
  group.
\newblock ar{X}iv:1808.04387v2, April 2019.

\bibitem{HR1959}
Seymour Haber and Azriel Rosenfeld.
\newblock Groups as unions of proper subgroups.
\newblock {\em Am Math Mon.}, 66(6):491--494, 1959.

\bibitem{HOMMES2006}
P.~E. Holmes.
\newblock Subgroup coverings of some sporadic groups.
\newblock {\em J. Combin. Theory Ser. A}, 113:1204--1213, 2006.

\bibitem{KNPS2016}
Luise-Charlotte Kappe, Daniela Nikolova-Popova, and Eric Swartz.
\newblock On the covering number of small symmetric groups and some sporadic
  simple groups.
\newblock {\em Groups Complexity Cryptology}, 8(2):135--154, 2016.

\bibitem{LAUDERDALE2013}
L.-K. Lauderdale.
\newblock Lower bounds on the number of maximal subgroups in a finite group.
\newblock {\em Arch. Math.}, 101(1):9--15, 2013.

\bibitem{LIEBECK-PRAEGER-SAXL1988}
Martin~W. Liebeck, Cheryl~E. Praeger, and Jan Saxl.
\newblock On the {O}'{N}an-{S}cott theorem for finite primitive permutation
  groups.
\newblock {\em J. Austral. Math Soc. (Series A)}, 44:389--396, 1988.

\bibitem{LUCIDO2003}
M.~S. Lucido.
\newblock On the covers of finite groups.
\newblock In {\em Groups St. Andrews}, volume 305 of {\em London Mathematical
  Society Lecture Note Series}, pages 395--399. Cambridge University Press,
  2003.

\bibitem{MAROTI2005}
Attila Mar\'{o}ti.
\newblock Covering the symmetric group with proper subgroups.
\newblock {\em J. Combin. Theory Ser. A}, 110(1):97--111, 2005.

\bibitem{OE2019}
Ryan Oppenheim and Eric Swartz.
\newblock On the covering number of ${S}_{14}$.
\newblock {\em Involve}, 12(1):89--96, 2019.

\bibitem{SCORZA1926}
G.~Scorza.
\newblock I gruppi che possono pensasi come soma di tre loro sottogruppi.
\newblock {\em Boll. Un. Mat. Ital.}, 5:216--218, 1926.

\bibitem{SCOTT1980}
Leonard~L. Scott.
\newblock Representations in characteristic $p$.
\newblock In {\em Proc. of Sympos. in Pure Math.}, volume~37, pages 318--331,
  Providence, RI, 1980. Santa Cruz conference on finite groups, Amer. Math.
  Soc.

\bibitem{SWARTZ2016}
Eric Swartz.
\newblock On the covering number of symmetric groups having degree divisible by
  six.
\newblock {\em Discrete Math.}, 339(11):2593--2604, 2016.

\bibitem{TOMKINSON1997}
M.~J. Tomkinson.
\newblock Groups as the union of proper subgroups.
\newblock {\em Math. Scand.}, 81:191--198, 1997.

\end{thebibliography}

\end{document}